\documentclass{article}
\usepackage{maa-monthly}
\usepackage{mathtools}
\usepackage{graphicx}
\usepackage{xfrac}
\usepackage[numbers]{natbib}
\usepackage{caption}
\usepackage{subcaption}
\usepackage{parselines} 
\DeclareMathOperator\invtanh{invtanh}
\DeclareMathOperator\csch{csch}
\begin{document}
\theoremstyle{definition}
\newtheorem*{definition}{Definition}
\newtheorem*{remark}{Remark}

\theoremstyle{theorem}
\newtheorem{theorem}{Theorem}

\newcommand{\invG}{\ensuremath{\mathrm{inv}\Gamma}}
\newcommand{\invGhat}{\ensuremath{<David's macro>}}
\newcommand{\gf}{$\Gamma$ function }
\newcommand{\stf}{Stirling's formula}
\newcommand{\gz}{\Gamma (z) }
\newcommand{\gzo}{\Gamma (z+1) }
\newcommand{\e}{\mathrm{e}}
\newcommand{\tipi}{t-i\pi}
\renewcommand\qedsymbol{$\hfill\natural$}

\DeclarePairedDelimiter{\ceil}{\lceil}{\rceil}
\title{Gamma and Factorial in the Monthly}
\markright{}
\author{Jonathan M. Borwein and Robert M. Corless}

\maketitle

 \section{Introduction} \label{introduction}

The Monthly has published roughly fifty papers on the $\Gamma$ function or Stirling's formula. We survey those papers (discussing only our favourites in any detail) and place them in the context of the larger mathematical literature on $\Gamma$. We will adopt a convention of $z \in \mathbb{C}$, $n \in \mathbb{N}$, $x \in \mathbb{R}$, and often $x>0$. We begin with notation:
\begin{equation}
 \Gamma (z) = (z-1)! = \int^{\infty}_{0} t^{z-1}\e^{-t} \ dt \phantom{due}\mathrm{for} \phantom{s} \operatorname{Re}z > 0 \>,
\end{equation} which extends the positive integer factorial $n!=n(n-1)\cdots 2 \cdot 1$ to the right half of the complex plane. The reflection formulas (discovered by Euler)
\begin{equation}
(-z)!z! = \frac{\pi z}{\sin \pi z} \phantom{due} \mathrm{or} \phantom{due} \Gamma(1-z)\Gamma(z) = \frac{\pi}{\sin \pi z}
\end{equation} explicitly show the analytic continuation and identify the poles at $z=-n$, nonpositive integers (for $\Gamma$), and negative integers (for factorial) with residues $(-1)^{n}/n!$ for $\Gamma$ and an equivalent shifted result for factorial.

We reluctantly bypass the amusing ``notation war" where authors argue about the ``minor but continual nuisance" of the shift by 1 in passing between factorial notation and $\Gamma$ notation, $z!=\Gamma(1+z)$. See for example the footnotes in~\cite{davis},~\cite{henrici1977},~\cite{jeffreys}, the introduction in~\cite{lanczos1964}, and the history in~\cite{gronau2003}, for some fullisades.

An important fact is that $\Gamma(x)$ is \textsl{logarithmically convex}, i.e. that $\ln \Gamma (x)$ is convex (see figure~\ref{fig:lngamma}) and that this distinguishes $\Gamma$ amongst all interpolants of the factorial. This is the Bohr-Mollerup-Artin theorem. A function that satisfies $f(z+1)=zf(z)$ and $f(1)=1$---but is not logarithmically convex---is called a \textit{pseudogamma function}. 

Euler's definition of $z!$ is, for $z$ not a negative integer, 
\begin{align} 
z! &= \lim\limits_{m\rightarrow\infty} \frac{m!(m+1)^z}{(z+1)(z+2)\cdots(z+m)} \>.
\end{align} We have $z! = z (z-1)!$ and $\Gamma (z+1) = z \Gamma (z)$. Weierstrass apparently preferred Newman's formula:
\begin{equation} \label{weierstrass}
\frac{1}{\gz} = z\e^{\gamma z} \prod_{k\geq 1} \bigg(1+ \frac{z}{k} \bigg) \e^{-z/k} \>,
\end{equation} where $\gamma=\lim\limits_{n\rightarrow \infty} (\sum_{k=1}^{n} \sfrac{1}{k}- \ln n)$ is the Euler-Mascheroni constant, $\gamma = 0.577 \ldots \>.$

The derivative of $\Gamma$ is less convenient than is the combination 
\begin{equation}
\psi(z) = \frac{d}{dz} \ln \gz = \frac{\Gamma'(z)}{\gz} \>.
\end{equation} The function $\psi(z)$ is also called the digamma function.

Legendre derived the duplication formula:
\begin{align} \label{eq:Legendre}
\Gamma(2z) &= (2\pi)^{-\frac{1}{2}}2^{2z-1/2}\gz\Gamma\left(z+\frac{1}{2}\right) \>. 
\end{align} 
The reflection formula 
\begin{align}
\zeta(z)&= \zeta(1-z)\Gamma(1-z)2^z\pi^{z-1} \sin\frac{\pi z}{2}
\end{align} was known to Euler, but was first proved by Riemann. Here $\zeta(z) = \sum_{n\geq 1} 1/n^z$ is the famous Riemann zeta function. Kummer (1847) derived a Fourier series:
\begin{align}
\ln\gz &= \frac{1}{2} \ln \frac{\pi}{\sin\pi z} + \sum_{k \geq 1} \frac{(\gamma +\ln 2\pi k ) \sin 2\pi k z}{\pi k}\>, \label{eq:kummer}
\end{align} which is illustrated in figure~(\ref{fig:kummer}). Also, the Beta function is
\begin{align}
B(s,t) = \int^{1}_{x=0} x^{s-1}(1-x)^{t-1} \ dx = \frac{\Gamma (s) \Gamma (t)}{\Gamma (s+t)} \>.
\end{align}

\begin{figure*}[t!]
    \centering
    \captionsetup{font=scriptsize,labelfont={bf}}
    \captionsetup{font=scriptsize}
    \begin{subfigure}[b]{0.5\textwidth}
        \centering
        \includegraphics[height=2.5in]{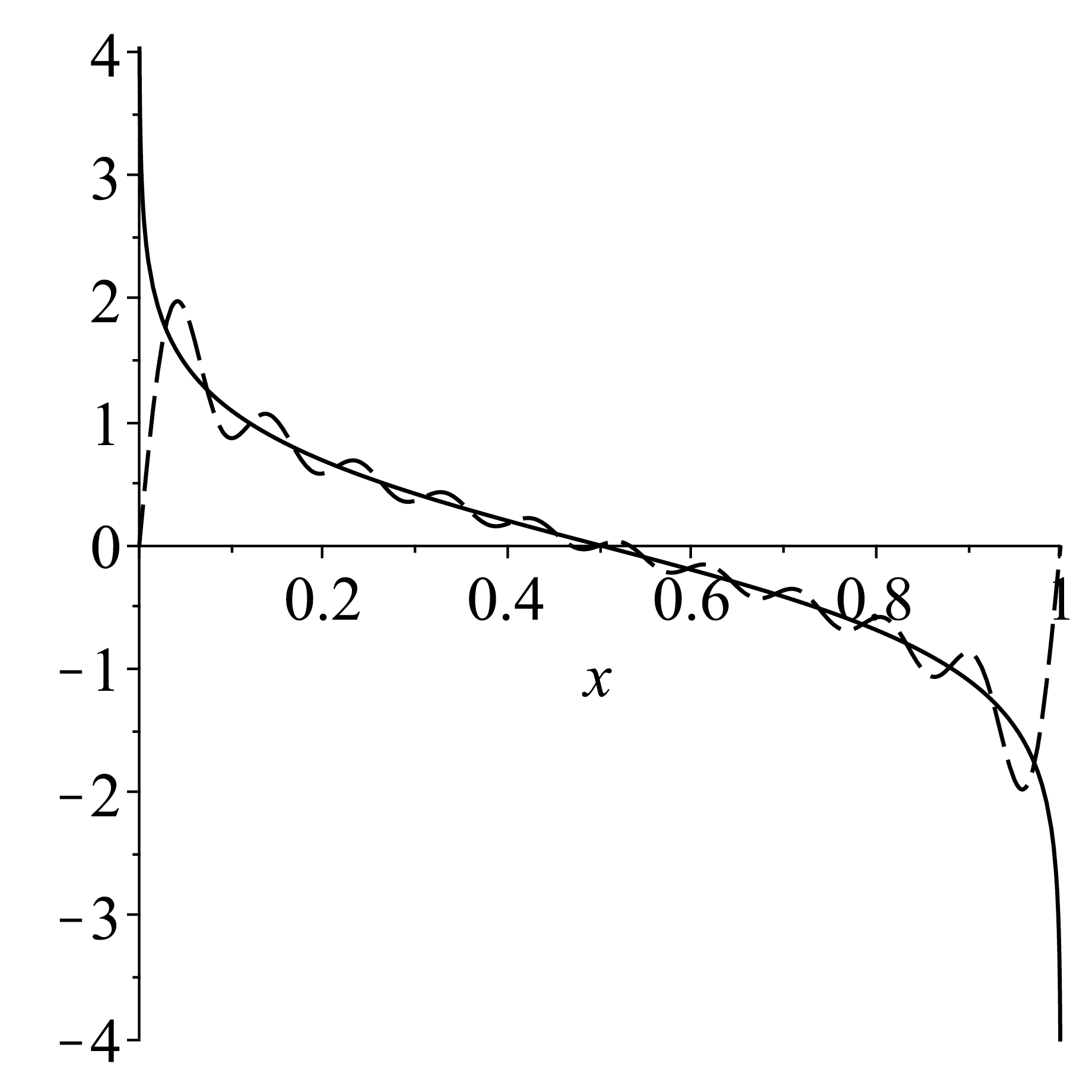}
        \caption{}
        \label{fig:kummer}
    \end{subfigure}%
    ~ 
    \begin{subfigure}[b]{0.5\textwidth}
        \centering
        \includegraphics[height=2.5in]{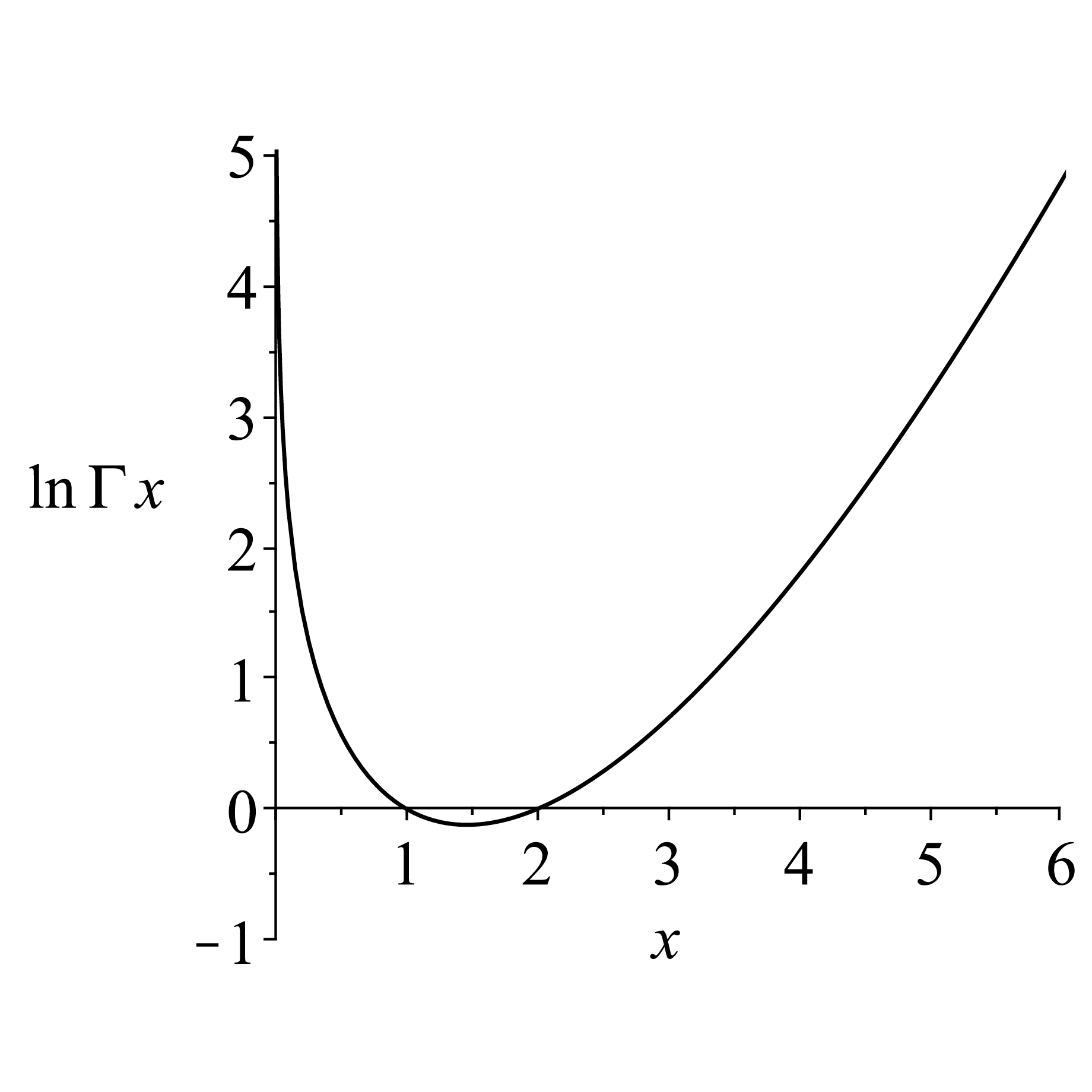}
        \caption{}
        \label{fig:lngamma}
    \end{subfigure}
    \caption{(a) A graph of $y=\ln ( \Gamma (x) \sqrt{(\sin \pi x)/\pi})$ (solid line) with a graph of $10$ terms of the Fourier Series from equation (\ref{eq:kummer}) (dashed line) superimposed. (b) A graph of $\ln \Gamma (x)$ for $x>0$. The Bohr-Mollerup-Artin characterization of $\Gamma$ as the unique log-convex function satisfying $\Gamma (x+1) = x \Gamma (x)$, $\Gamma (1) = 1$ is illustrated here. The convexity of $\ln \Gamma (x)$ is clearly visible.}
\end{figure*}

\section{The $\Gamma$ and Factorial Functions: A cross--section of mathematics}\label{sec:gammaitself}

The following Monthly papers (in alphabetical order by surname of the first author) are the ones that stand out, our ``eleven favourites"\footnote{In~\cite{borwienchapman} the authors highlighted what they called ``$\pi *$" articles, i.e. those that had been cited at least $30$ times. In contrast, inclusion in this list here is purely subjective.}:

\begin{enumerate}

\item Patrick Ahern and Walter Rudin, ``Geometric Properties of the $\Gamma$ function"~\citep{ahern}
\item Wladimir de Azevedo Pribitkin, ``Laplace's Integral, the Gamma Function and Beyond"~\cite{prihitkin}
\item Richard Askey, ``Ramanujan's extensions to the $\Gamma$ and Beta functions"~\cite{askey}
\item Bruce Berndt, ``The $\Gamma$ function and the Hurwitz $\zeta$-function" ~\cite{berndt} 
\item Manjul Bhargava, ``The factorial function and generalizations"~\cite{bhargava} 
\item Philip J. Davis, ``Leonhard Euler's Integral"~\cite{davis}
\item Louis Gordon, ``A stochastic approach to the $\Gamma$ function"~\cite{gordon}
\item Detlef Laugwitz and Bernd Rodewald, ``A simple characterization of the $\Gamma$ function"~\cite{laugwitz} 
\item Reinhold Remmert, ``Wielandt's theorem about the $\Gamma$ function"~\cite{remmert}
\item Lee Rubel, ``A survey of transcendentally transcendental functions" ~\cite{rubel}
\item Gopala Krishna Srinivasan, ``The Gamma function: An eclectic tour"~\cite{srinivasan} 
\end{enumerate}

These eleven disparate papers provide a cross--section of mathematics: analysis, geometry, statistics, combinatorics, logic and number theory. We begin with the earliest, the Chauvenet--winning~\cite{davis} by Philip J. Davis. We then tour through the others and end with Fields medallist Manjul Bhargava's~\cite{bhargava}.

\subsubsection*{Fundamental Properties}

Davis tells the story of $\Gamma$ as an instance of how mathematics evolves. This required informed historical commentary and an appreciation of mathematical aesthetics. The paper was dedicated to the memory of Milton Abramowitz, who died (too young) in 1958. Davis was then chief of Numerical Analysis at The National Bureau of Standards of the United States, and a contributor to the great project that became (after six years of further hard work by Irene A. Stegun) the monumental Handbook of Mathematical Functions~\cite{abramowitz1966}. Indeed the graph of the real $\Gamma$ function in~\cite{davis} is apparently the same as in the Handbook. This is not a surprise: Davis wrote that section.

This monumental work has recently been updated as The Digital Library of Mathematical Functions (http://dlmf.nist.gov)~\cite{olver}. The DLMF and the similar-but-different INRIA project, Bruno Salvy's Dynamic Dictionary (http://ddmf.msr-inria.inria.fr/1.9.1/ddmf)~\cite{salvy}, fill a need for good online mathematical material. 

To return to $\Gamma$, the paper~\cite{davis} gives a story that enriches the plain facts in the Handbook, and illustrates the evolution of mathematical thought. Davis spent a lot of space on the interpolation of $n!$ by $\Gamma (z+1)$, treating it with the philosophical attention it deserves. He did, however, note that $\Gamma$ arises in applications mostly because the integral defining it does.

The impressive paper~\cite{gordon} by L. Gordon proves many of these same facts by considering the $\Gamma$ distribution and taking the statistical point of view. Of course, $\Gamma$ is widely used in probability and statistics. 

Now we come to the survey~\cite{rubel} by L. A. Rubel, which is concerned with $\Gamma$ only insofar as $\Gamma$ is the most famous of the ``transcendentally transcendental" (TT) functions; that is, functions that do not satisfy any polynomial differential equation $P(x;y,y',\cdots,y^{n})=0$. Functions that \textit{do} satisfy such equations are called Differential Algebraic functions; for instance, the Lambert $W$ function~\cite{cghjk} satisfies $y\e^{y}=x$ and thus is DA because
\begin{equation}
x(1+y)y' - y = 0\>.
\end{equation} It's mentioned in~\cite{davis} that the first proof that $\Gamma$ is TT was due to H\"older in 1887.

Rubel economically and intelligibly presents Ostrowski's proof that $\Gamma$ is TT. The proof has many of the characteristics one sees in computer algebraic proofs: ordering terms in reverse lexicographic order (Rubel doesn't call it that), and using divisibility arguments to arrive at a minimalist polynomial $P$ for $\Gamma$, which can then be shown self-contradictory. The only property of $\Gamma$ aside from differentiability that is used is the functional equation $\Gamma (x+1) = x \Gamma (x)$. Therefore, the proof runs through unchanged for (sufficiently differentiable) pseudogamma functions. In particular, Hadamard's entire pseudogamma $H (z)$ (see equation~\eqref{eq:Hadamard} below), is TT. This was not pointed out in~\cite{davis} and it's not obvious from its definition (after all, the combination $\Gamma (x+1) / \Gamma (x)$ is algebraic; it's not automatic that combinations of $\Gamma ( \sfrac{(1-x)}{2}),  \Gamma (1-\sfrac{x}{2})$ and $\Gamma (1-x)$ will be TT). Altogether this highly-cited paper provides a very readable introduction to the area, and fills a gap in the education of people who use $\Gamma$.

\subsubsection*{Characterizing $\Gamma$}

The paper~\cite{laugwitz} by Laugwitz and Rodewald was a suprise to RMC, as was the paper~\cite{remmert} by Remmert. RMC had not been aware of satisfactory alternatives to the Bohr-Mollerup-Artin characterization of $\Gamma (x)$. Such a condition is necessary, because there are pseudogamma functions; for instance, Davis gives Hadamard's example (``Hadamard's pseudogamma function''):
\begin{equation}
H(z) = \frac{1}{\Gamma (1-z)} \frac{d}{dz} \ln \Bigg[ \Gamma \bigg( \frac{1-z}{2} \bigg) \bigg/ \Gamma \bigg( 1 - \frac{z}{2} \bigg) \Bigg] \>.
\end{equation} Maple simplifies this quite dramatically to 
\begin{equation} \label{eq:Hadamard}
H(z) = \frac{\psi \bigg(1-\frac{z}{2} \bigg) - \psi \bigg( \frac{1-z}{2} \bigg)}{2 \Gamma (1-z)} \>,
\end{equation} where $\psi (z) = \Gamma'(z) / \Gamma (z)$. In retrospect, it's not so dramatic, merely replacing a derivative with $\psi$; nonetheless it \textit{looks} simpler. For a more nuanced discussion of the satisfaction given by such formulae, see~\cite{wilf} by Wilf. 

\begin{figure}[htb!]
	\center
	\includegraphics[scale=0.4]{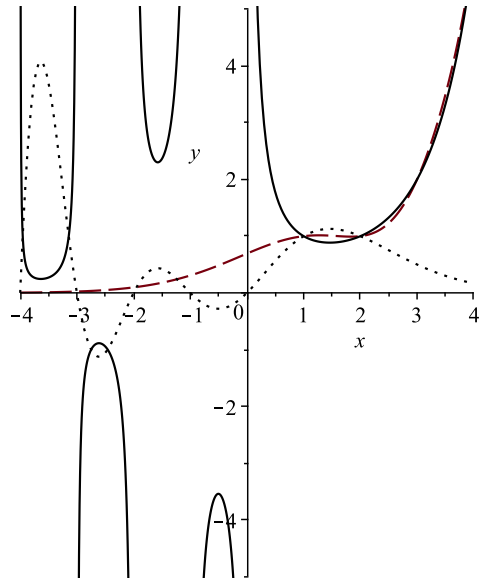}
	\caption{On the same scale as the graph in~\cite{davis} we plot $\Gamma (x)$ (solid black line), $1/\Gamma (x)$ (dotted line) and Hadamard's pseudogamma function from equation~(\ref{eq:Hadamard}) (dashed line). The lack of convexity of Hadamard's function is clearly visible.}
	\label{fig:realgamma}
\end{figure}

Curiously enough $H(z)$ is an \textit{entire} function. Why, then, is $\Gamma (z)$ to be preferred to $H(z)$, given that $\Gamma$ has poles at non-positive integers? The  ``real" answer given by Davis is that the integral for $\Gamma$ appears frequently. But the logarithmic convexity characterization of Bohr-Mollerup-Artin is compelling. Davis eloquently explains \textit{``The desired condition was found in notions of convexity. A curve is convex if the following is true of it: take any two points on the curve and join them by a straight line; then the portion of the curve between the points lies below the line. A convex curve does not wiggle; it cannot look like a camel's back. At the turn of the century, convexity was in the mathematical air$\ldots$"}~\cite[p.~867]{davis}. 

We point out that logarithmic convexity is stronger than mere convexity: the convexity of $f(x)$ does not imply the convexity of $\ln f(x)$, but the convexity of $\ln f(x)$ implies that of $f(x)$.

But there are alternative characterizations. The characterization discussed by Laugwitz and Rodewald in~\cite{laugwitz} is actually due to Euler himself, and is very elegant: ``For a fixed infinitely large natural number $n$ and all finite natural numbers $m$, the expression $(n+m)!$ behaves approximately like a geometric sequence."~\cite[p.534]{laugwitz}. Laugwitz and Rodewald go on to put this in a more usual `limit' language: they prove that if $L(0)=0, L(x+1) = \ln (x+1) + L(x)$, and $L (n+x) = L(n) + x \ln (n+1) + r_{n} (x)$ where $r_{n} (x) \to 0$ as $n \to \infty$ (here $x\geq 0$ and $n \in \mathbb{N}$) then $L (x)$ exists and is unique and is thus $\ln (\Gamma (x+1))$. They do this by showing that the conditions imply a convergent series for $L(x)$; they then verify that the resulting series,
\begin{equation} \label{eq:13}
L(x) = -\gamma x + \sum_{k \geq 1} \Bigg( \frac{x}{k} - \ln \bigg( 1 + \frac{x}{k} \bigg) \Bigg) \>,
\end{equation} satisfies the conditions.

Another alternative characterization is apparently due to Wielandt. Until Remmert's paper~\cite{remmert}, this seems to have been ``hardly known". Wielandt's theorem states that if $F(z)$ is holomorphic in the right half plane, $F(z+1) = z F(z)$ for all such $z$, and is \textsl{bounded} on the strip $1 \leq  \operatorname{Re}z \leq 2$, then $F(z)=a \Gamma (z)$ and $a=F(1)$. This characterization seems to be quite easy to use, more so than logarithmic convexity. 

The intriguing and concise paper~\cite{ahern} by Ahern and Rudin cites~\cite{remmert} but does not appear to actually use Wielandt's theorem; instead it extends logarithmic convexity into the complex plane. The final theorem in the paper states that $\ln \Gamma (z)$ is univalent in $\operatorname{Re} \ z > x_{0}$, where $x_{0}$ is the only positive zero of $\psi (x)$. They note $1<x_{0}<2$; by a computer algebra system, $x_{0}=1.4616321\ldots$. There seems to be no standard name for this number. This result on univalence seems to be as close as it gets in the Monthly to a discussion of the functional inverse of $\Gamma$ (see later in this section).

The very interesting (and fun to read) paper~\cite{prihitkin} by W. de Azevedo Pribitkin \textit{does} use Wielandt's theorem. The paper is concerned with the Laplace integral for the reciprocal of $\Gamma$:
\begin{equation}
\frac{1}{\Gamma (s)} = \frac{1}{2 \pi} \int^{\infty}_{\mu = - \infty} \frac{e^{a+i \mu}}{(a+i \mu)^{s}} d\mu
\end{equation} where $a>0$ is any positive real number. This paper comes as close as any in the Monthly to the computation scheme of ~\cite{schmelzer}.

Srinivasan's ``Eclectic Tour"~\cite{srinivasan} also uses Wielandt's theorem. This wide-ranging survey achieves a remarkable complementarity to Davis' paper; indeed even the historical details are different (not contradictory, just complementary). New elegant proofs of many results are attained using a function-theoretic mindset and a systematic approach, apparently termed the ``additive approach" by Remmert. The paper starts with the differential equation (which is also in~\citep{davis})
\begin{equation}\label{eq:alsoInDavis}
\frac{d^{2}}{dz^{2}} \ln f(z) = \sum_{n\geq 0} \frac{1}{(n+z)^{2}} \>.
\end{equation} With initial conditions $f(1)=1$ and $f'(1)=-\gamma$, this gives $f(z)= \Gamma (z)$. Wielandt's theorem is used in the proof. Altogether, this is a remarkable paper, which presents (with proofs!) a very significant body of knowledge of $\Gamma$, and of related functions. 

\subsubsection*{Relationship to $\zeta$}

Bruce Berndt's short paper~\citep{berndt} illuminates old connections of $\Gamma$ to the Hurwitz~$\zeta$ function. He gives several proofs using that connection. For instance, he proves equation~(\ref{eq:kummer}). He first proves 
\begin{equation}
\ln \gz = \zeta'(0,z) - \zeta'(0)
\end{equation} where $'$ means differentiation with respect to $s$ (the first variable) and 
\begin{equation}
\zeta(s,z) = \sum_{k\geq 0} (k+z)^{-s}
\end{equation} is the Hurwitz zeta function; $\zeta(s) = \zeta(s,1)$ is the Riemann zeta function.

\subsubsection*{Generalizations}

We begin with Askey's beautiful paper~\cite{askey}. Askey concurred with Davis on the reason $\Gamma$ is useful: ``Euler's integral (1.3) occurs regularly and is the real reason for studying the gamma function."~\cite[p.~347]{askey}. Other than that, Askey's paper is quite different from Davis's, although it, too, tells a story, namely the story of Ramanujan's $q$-extensions of the $\Gamma$ and Beta functions. Askey provides much detail, many results, and concise proofs. He places the results in context, and explains why they are important. His starting point was that Hardy thought the results strange; Askey demonstrates that the results are profoundly connected to modern developments, and thus not so `strange' today. Of course, part of the direction of modern developments is  due to Ramanujan, so this conclusion might have been expected. Both Davis and Askey show relationships of $\Gamma$ to other functions and areas.

Now we come to the paper~\cite{bhargava} by Manjul Bhargava. The first time RMC read the paper, RMC didn't notice the author's name\footnote{In his defence, it was in a BIG pile of $\Gamma$ papers!}; RMC noted down that the paper was concerned with discrete generalizations of the factorial, was focused on number theoretic ideas and applications, and seemed quite substantial, much more so than some other ``generalization" papers in the Monthly, namely~\cite{newton} and~\cite{fort} (although the last isn't so lightweight, either).

On RMC's next pass, he paid a bit more attention. Bhargava was awarded the Fields medal in 2014 for new methods in geometry of numbers. It seems a safe bet that this paper, although written 14 years previous to his prize, might perhaps contain something a bit more substantial than the average paper in the Monthly on the $\Gamma$ function\footnote{It's interesting that the notion ``an average paper in the Monthly on $\Gamma$" actually makes sense.}. Of course, it does, even though it is written in a lively, almost joyful, style. 

Instead of making a small change in the definition of $\Gamma$ (in contrast, the paper~\cite{newton} by T. A. Newton changes integrals to sums, and derivatives to differences; the paper~\cite{fort} by Tomlinson Fort changes the ``$+1$" in the functional equation to a variable ``$+h_{n}(x)$") Bhargava isolates an invariant that characterizes the factorial: he writes $k!$ as the product over all primes $p$ of certain valuations. Then he is able to generalize this to work over subsets of the integers or even over more abstract objects. To do this he used something reminiscent of the Leja ordering for good-quality use of Newton interpolation, and indeed the bases he constructs are very like Newton interpolational bases. This generalization encompasses the $q$-factorials from enumerative combinatorics which we met in~\cite{askey} and it seems the ones in~\cite{newton} too, as well as others in the non-Monthly literature. 

Bharghava ends his paper with a sequence of then-open questions, including questions on generalizations of this factorial to a generalized $\Gamma$ function, and on generalizations of Stirling's formula.

\subsection{Shorter papers and notes}

There are quite a few short pieces on $\Gamma$ in the Monthly. One of the oldest papers in the Monthly is~\cite{basoco}, which uses explicit expressions for orthogonal polynomials and their orthogonality conditions to arrive at certain identities for binomial double sums. 

The Note~\cite{wendel} uses H\"older's inequality to establish
\begin{equation}
\bigg( 1 - \frac{a}{x+a} \bigg)^{1-a} \leq \frac{\Gamma (x+a)}{x^{a} \Gamma (x)} \leq 1
\end{equation} for $0<a<1$ and $x>0$; letting $x \to \infty$ gives the classical $\Gamma (x+a) \sim x^{a} \Gamma (x)$, which is an asymptotic generalization of the functional equation $\Gamma (x+1)=x \Gamma (x)$. Nanjundiah's paper~\cite{nanjundiah} contains an interesting trick, using the asymptotic relation $B(\alpha,t) \sim \Gamma(\alpha)/t^{\alpha}$ as $t \to \infty$ between the Beta function and the $\Gamma$ function to get the exact relationship, $B(x,y) = \Gamma(x)\Gamma(y)/\Gamma(x+y)$ and from there through $\Gamma(x+a) \sim x^{a} \Gamma(x)$ (again!) to Euler's product definition.

The very pretty paper~\cite{nijenhuis} cites~\cite{borweinzucker} and a Monthly problem~\cite{glasser} as motivation, and produces a recipe for finding some short products of Gammas that give simple values, such as
\begin{equation}
\Gamma \bigg( \frac{1}{62} \bigg) \Gamma \bigg( \frac{33}{62} \bigg) \Gamma \bigg( \frac{35}{62} \bigg) \Gamma \bigg( \frac{39}{62} \bigg) \Gamma \bigg( \frac{47}{62} \bigg) = 2^{4}\pi^{5/2} \>.
\end{equation} The paper~\cite{brookes} merely finds $\Gamma$ functions in a particular solution to the strange (nonlinear) ODE $\frac{d^{n}y}{dx^{n}} \frac{d^{m}x}{dy^{m}}=1$.

The paper~\cite{knuth} by Knuth probably shouldn't be included in this review, although the ratio $n!/n^{n}$ plays a starring role: this is Egorychev's theorem, that the permanent (like the determinant, but with all plus signs instead of $(-1)^{\sigma}$ in the permutation expansion) of a doubly stochastic matrix cannot be less than this, and therefore the doubly stochastic matrix $\mathbf{A}$ with $a_{ij}=1/n$ has this minimal permanent. Egorychev had settled a long-standing conjecture of P\'{o}lya's with this theorem, and the purpose of~\cite{knuth} was pure exposition. The paper also demonstrates that this review cannot be expected to be complete---the paper was not found by computer search, but rather by browsing, i.e. by serendipity. 

Speaking of serendipity, the paper~\cite{bRoss} uses an exploration of something related to fractional integration to discuss just that, in mathematical practice. The notion of fractional integration, i.e. interpolation in the dimension of a multiple integral, is an idea of the same type that led to the creation of $\Gamma$, of course, as discussed in~\cite{davis}.

The paper~\cite{frame} is concerned with the ratio and its approximation
\begin{equation}
\frac{\Gamma \bigg( n + \dfrac{1+u}{2} \bigg)}{\Gamma \bigg(n + \dfrac{1-u}{2} \bigg)} \sim \bigg( n^{2} + \dfrac{1-u^{2}}{12} \bigg)^{u/2}
\end{equation} for large $n$. This formula is exact when $n\geq 1$ is an integer and $u = 0, \pm 1, \pm 2$. The rest of the paper is concerned with the error, which the author shows is exp$(- \epsilon_{n}(u))$ with
\begin{equation}
\epsilon_{n}(u) = \frac{u (1-u^{2})(4-u^{2})}{6! n^{4}}F_{n}(u)
\end{equation} where $0<F_{n}(u) \leq 1$ for $|u| \leq 1$. The paper~\cite{jameson} is also concerned with $\Gamma$ function ratios. 

The expression~(\ref{eq:alsoInDavis}) has a similar analogue in rising powers, namely
\begin{align}
\frac{d^{2}}{dz^{2}} \ln \Gamma (z) &= \frac{1}{z} + \frac{1!}{2} \cdot \frac{1}{z(z+1)} + \frac{2!}{3} \cdot \frac{1}{z(z+1)(z+2)} + \cdots \\
&= \sum_{\ell \geq 0} \frac{\ell !}{(\ell+1)} \cdot \frac{1}{z^{\overline{\ell}}} \phantom{due to} \operatorname{Re}z >0
\end{align} where $z^{\overline{\ell}} = z(z+1)\cdots(z+\ell-1)$, $z$ to the $\ell$ rising. This equation can be derived from a formula of Gauss for a hypergeometric function with unit argument,
\[
F
\left(
\begin{matrix}
a, b \hfill \\ c
\end{matrix}
\, \middle\vert \,
1
\right) 
= \sum_{j \geq 0} \frac{a^{j}b^{j}}{c^{j}} \cdot \frac{1}{j!} = \frac{\Gamma (c) \Gamma (c-a-b)}{\Gamma (c-a) \Gamma (c-b)}
\] if $\operatorname{Re}(c-a-b) >0$ and $c \not \in \mathbb{N}$. The paper~\cite{ruben1976} gives some details of a proof. 

The paper~\cite{huber} uses the $\Gamma$ function to find volumes of spheres in higher dimension.

The paper~\cite{davis1957} by H. T. Davis considers the generating function for so--called \textit{logarithmic numbers} $L_{n}$: \[ \frac{x}{\ln (1+x)} = 1 + L_{1}x + L_{2}x^{2} + \cdots \] and establishes the curious interpolation formula $L(t)=L_{n}$ when $t=n$, where
\begin{equation}
L(t) = \int^{1}_{s=0}
\left( 
\begin{matrix}
s \hfill \\ t
\end{matrix}
\right) 
ds \>,
\end{equation} among other results. H. T. Davis does not use the binomial notation, but rather a ratio of $\Gamma$ functions, which makes clear the meaning of the neater binomial coefficient used here. 

Finally, a referee points out the delightful paper~\cite{lange1999} in which Lange uses properties of $\Gamma$ to give a new continued fraction for $\pi$. Since $\Gamma$ didn't appear in the title, we missed that one!

\subsection{Gaps: computation, visualization, inverse $\Gamma$}

One purpose of a survey such as this is to identify gaps. We noticed three. There's not much visualization for $\Gamma$ in the Monthly, apart from in~\cite{davis}. There's almost \textit{no} computational work. The most surprising omission, though, is the lack of exploration of the functional inverse of $\Gamma$, which we denote here by~$\invG$ or~$\check{\Gamma}$. To be fair, we have found only four references in total on $\invG$, and none in the Monthly. Still, to have studied a function since 1730, or 1894 from the start of the Monthly, without thinking of the inverse? It seems quite the blind spot.

\subsection{Visualization}

There has been significant progress in comprehensible visualization of complex functions, by the use of color phase plots~\cite{wegert2012,wegertsemmler2010}. In the inset figure in figure~\ref{fig:phasekey} the $z$-plane is colored by arg$(z)=\mathrm{invtan}(y,x)$; the same colors are used in figure~\ref{fig:phaseplot} for coloring by arg$(\Gamma (z))$. Of course, $z=x+iy$. The poles at $0$, $-1$, $-2$, and so on are clearly visible. In figure~\ref{fig:3dplot} we've added the same phase coloring to the plot of $|\Gamma(z)|$, on the same scale as the famous hand-drafted plot of~\cite{jahnkeemde1945}. If one looks carefully, this figure illustrates Wielandt's boundedness characterization: because the contours in the right half bend away from $y$ axis, the boundedness on (say) $1 \leq x \leq 2$ is entirely believable.

\begin{figure*}[t!]
    \centering
    \captionsetup{font=scriptsize,labelfont={bf}}
    \captionsetup{font=scriptsize}
    \begin{subfigure}[b]{0.5\textwidth}
        \centering
        \includegraphics[height=2.5in]{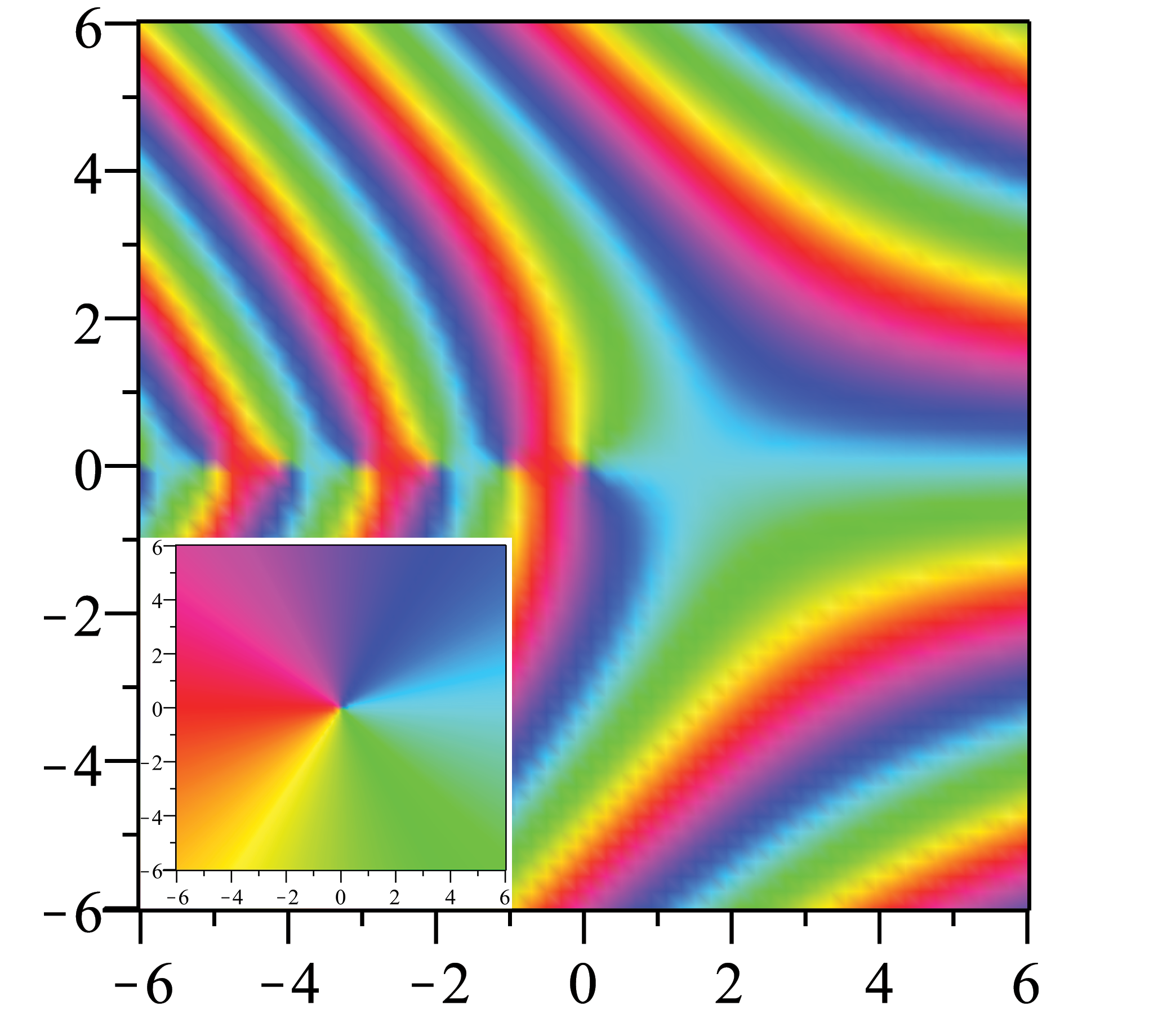}
        \caption{}
        \label{fig:phasekey}
        \label{fig:phaseplot}
    \end{subfigure}%
    ~ 
    \begin{subfigure}[b]{0.5\textwidth}
        \centering
        \includegraphics[height=2.5in]{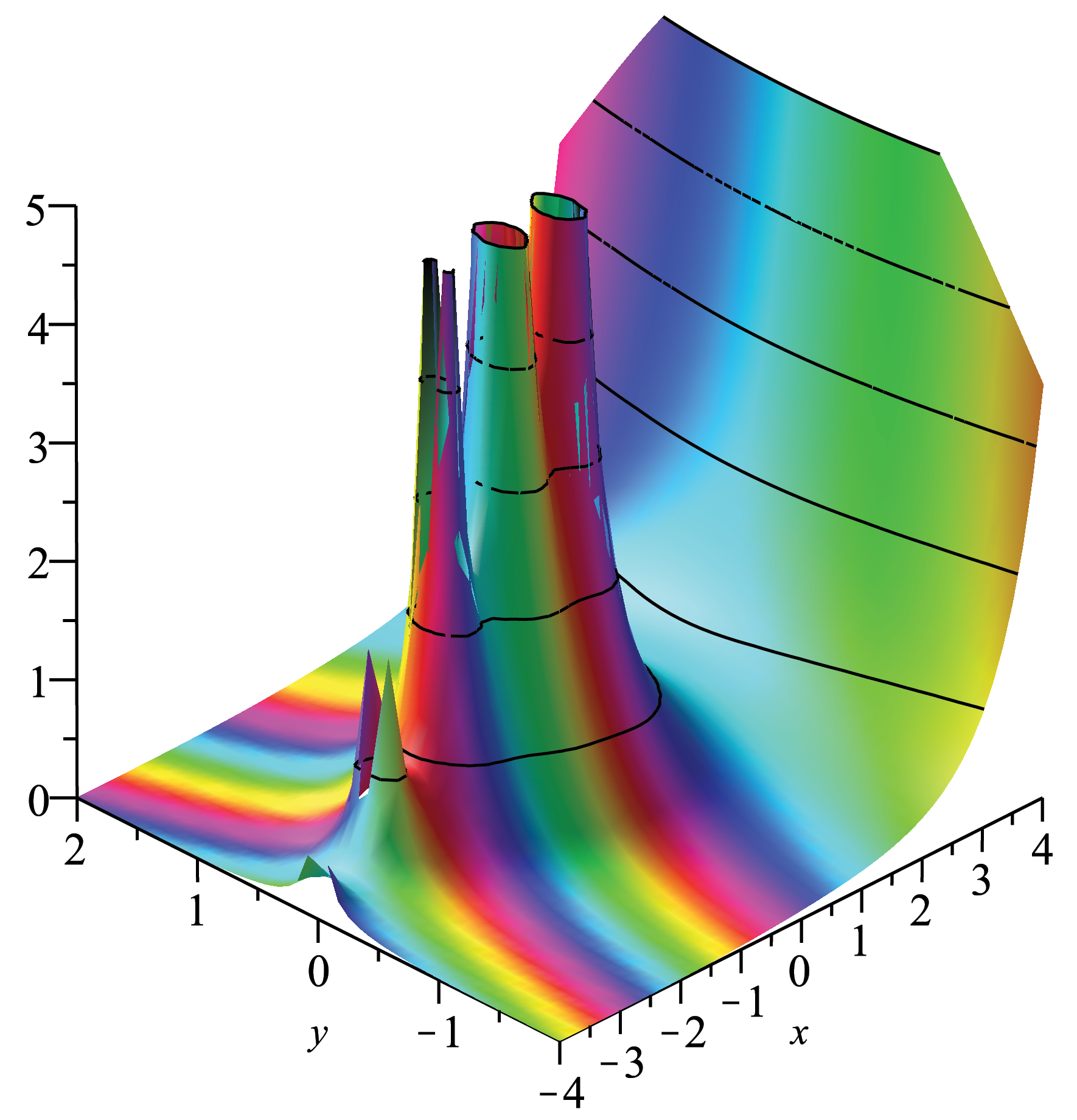}
        \caption{}
        \label{fig:3dplot}
    \end{subfigure}
    \caption{(a) A phase plot of $\Gamma (z)$ using the methods of~\cite{wegert2012}. The colors indicate the argument of $\Gamma (x+iy)$, and because (compared to the inset figure which plots just $\arg z $ where orange-yellow-green-blue-velvet-red in a rainbow palette goes counterclockwise around $z=0$) the colors change in the clockwise direction around $z=-n$ for integers $n \ge 0$, we can see that there are poles there. (b) The famous 3D plot of $|\Gamma(z)|$ from~\cite{jahnkeemde1945} can scarcely be improved upon. We add phase color here; any improvement on the older work seems merely cosmetic.}
\end{figure*}

\begin{figure*}[t!] \label{fig:figure4}
    \centering
    \captionsetup{font=scriptsize,labelfont={bf}}
    \captionsetup{font=scriptsize}
    \begin{subfigure}[b]{0.5\textwidth}
        \centering
        \includegraphics[height=2.5in]{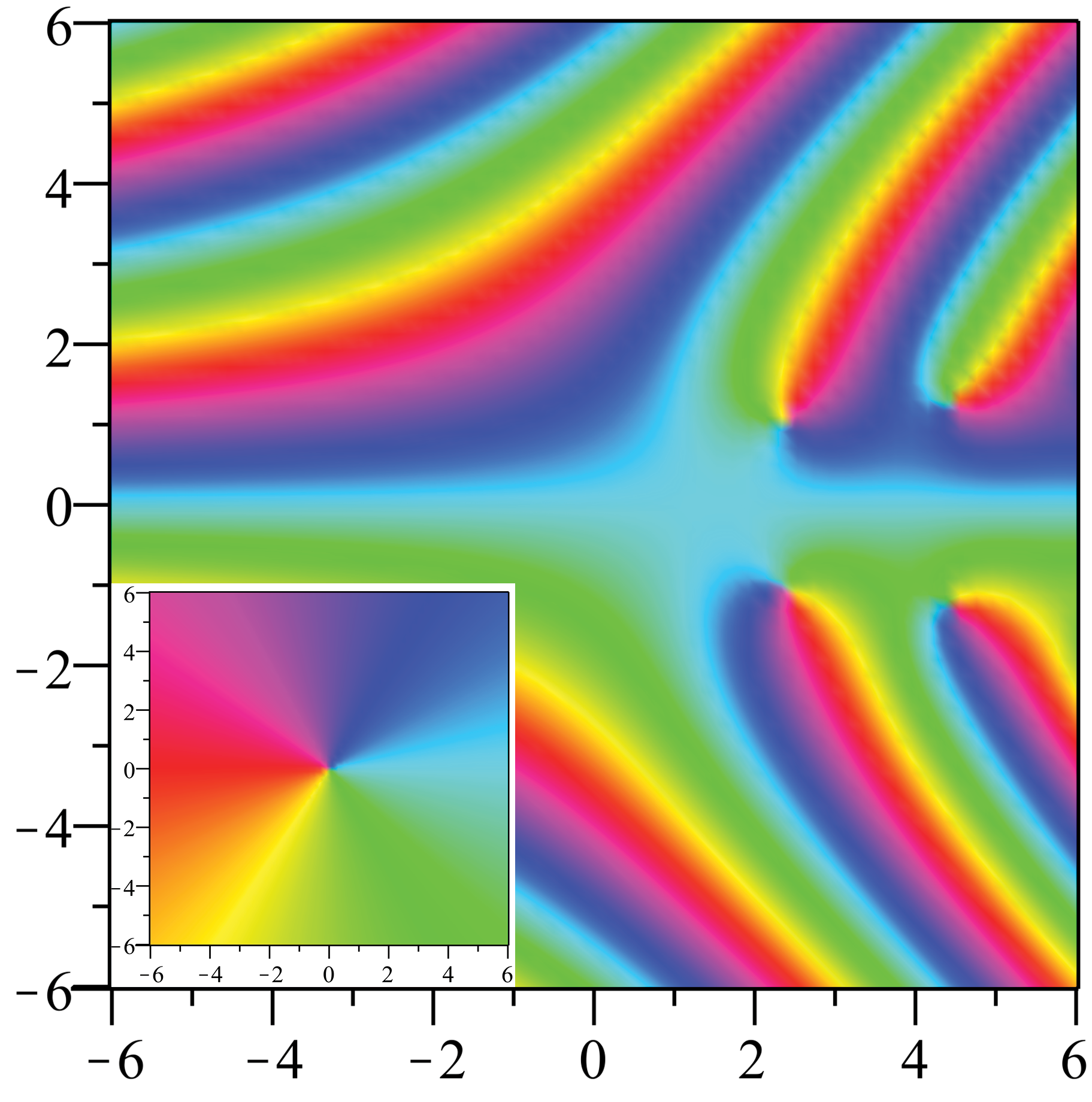}
        \caption{}
        \label{fig:hadamardsphaseplot}
    \end{subfigure}%
    ~ 
    \begin{subfigure}[b]{0.5\textwidth}
        \centering
        \includegraphics[height=2.5in]{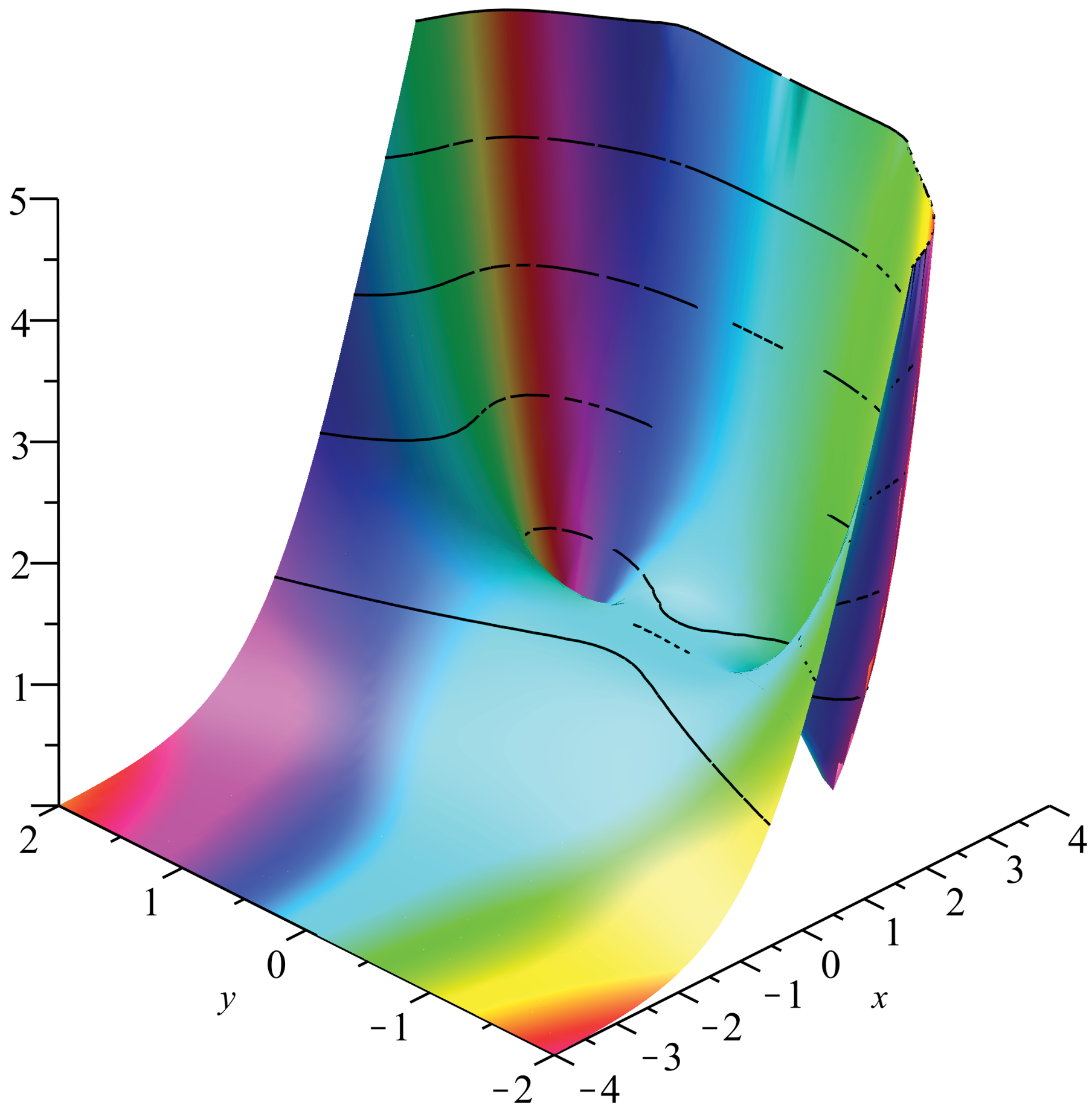}
        \caption{}
        \label{fig:hadamards3dplot}
    \end{subfigure}
    \caption{(a) A phase plot of Hadamard's pseudogamma function from equation~(\ref{eq:Hadamard}). We detect zeros in the right half plane where the colors as in the inset figure go clockwise. Verifying this numerically, we find for instance using Maple that $H(2.34652 \pm 1.05516i)\doteq 0$. (b) This figure shows $|H(z)|$ from equation~(\ref{eq:Hadamard}). It is colored by the phase. Because the contours bend left, this suggests that $H(z)$ is not bounded in strips and thus does not satisfy Wielandt's characterization. In fact we know from~\cite{fuglede} that its growth in (say) $1<\operatorname{Re}(z)<2$ as $\operatorname{Im}(z) \to \infty$ must be faster than exponential.}
\end{figure*}

A similar pair of plots are given for equation~\eqref{eq:13}, in figure 4. We will see more visualizations shortly, including one for a portion of the Riemann surface for $\invG$.

\subsection{Computation}\label{sec:computation}

We did not find in the Monthly any substantive discussion (or even pointers) as to how to compute the $\Gamma$ function, apart from using the fundamental recurrence to push $z$ far enough into the right half-plane that Stirling's asymptotic formula can be used. For a solid discussion of good methods, see~\cite{lanczos1964},~\cite{schmelzer}, and~\cite{spouge1994}. Schmelzer and Trefethen's~\cite{schmelzer} uses the beautiful contour integral 
\begin{equation}
 \frac{1}{\gz} = \frac{1}{2\pi i}\int_{C} s^{-z}\e^s ds \>.
\end{equation} Here the contour C, called a deformed Bromwich contour, winds around the negative real axis counterclockwise. It turns out that the trapezoidal rule (or the midpoint rule) is \textsl{spectrally accurate} for this quadrature, and the contour integral is (almost) entirely practical! See~\cite{schmelzer} for details. \\

\subsection{Inverse $\Gamma$}\label{sssec:num1}

\subsubsection{Plotting the real inverse Gamma function}

As stated earlier, it's surprising in hindsight that there's so few (and only so recent) studies on the inverse Gamma function, $\invG$, and none in the Monthly. We begin with the simplest things. Plotting inverse functions with computer software is very simple, even if there are no routines for the inverse function itself. The basic idea is explained in~\cite{snapper}: just exchange $x$ and $y$. Figure~\ref{fig:realinvG} was produced using Maple. The details of the code are available at \url{http://www.apmaths.uwo.ca/~rcorless/frames/Gamma/GammaBorweinCorless.mw}.

\begin{figure*}[t!]
    \centering
    \captionsetup{font=scriptsize,labelfont={bf}}
    \captionsetup{font=scriptsize}
    \begin{subfigure}[b]{0.5\textwidth}
        \centering
        \includegraphics[height=2.5in]{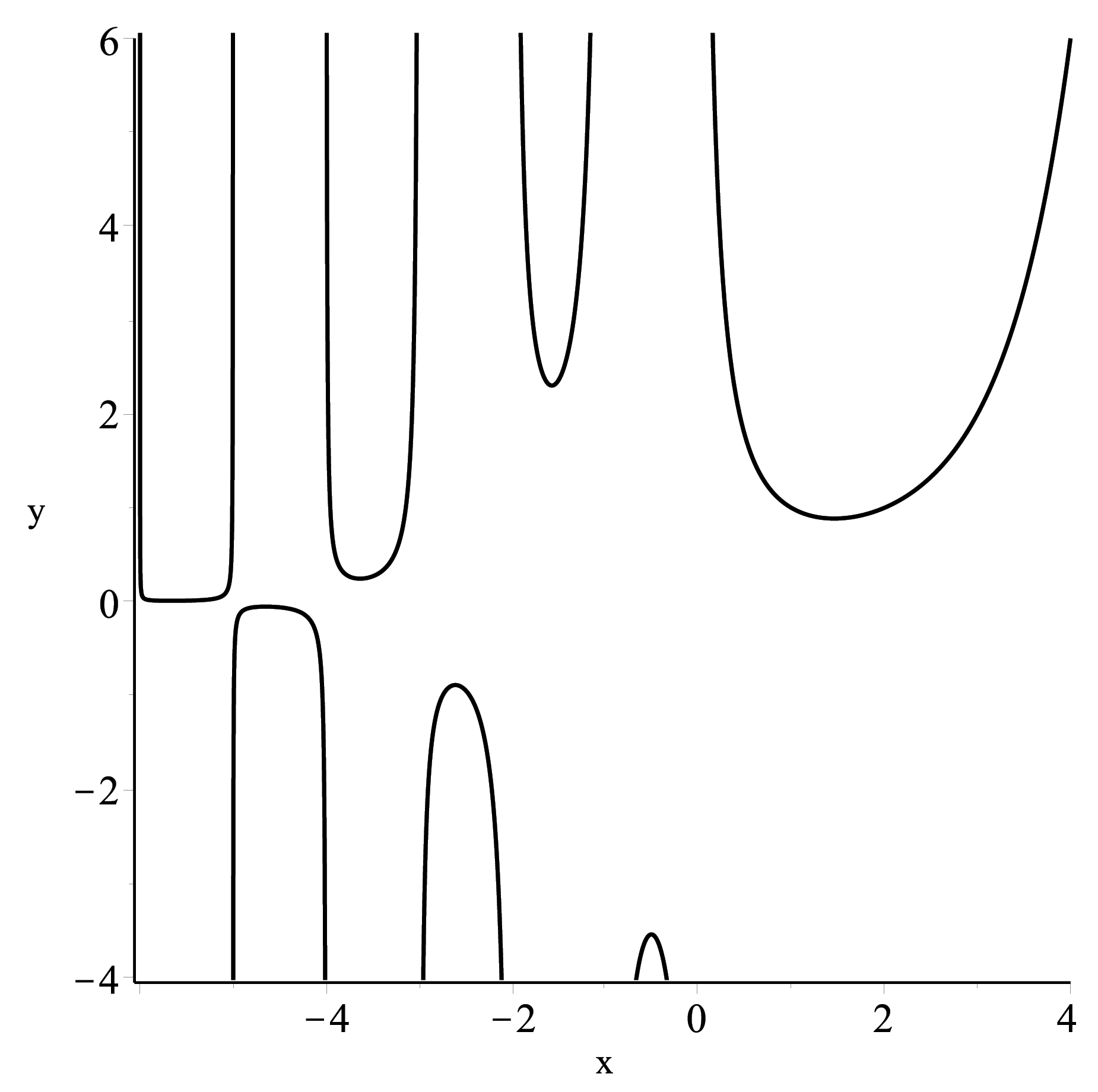}
        \caption{}
        \label{fig:realgammapara}
    \end{subfigure}%
    ~ 
    \begin{subfigure}[b]{0.5\textwidth}
        \centering
        \includegraphics[height=2.5in]{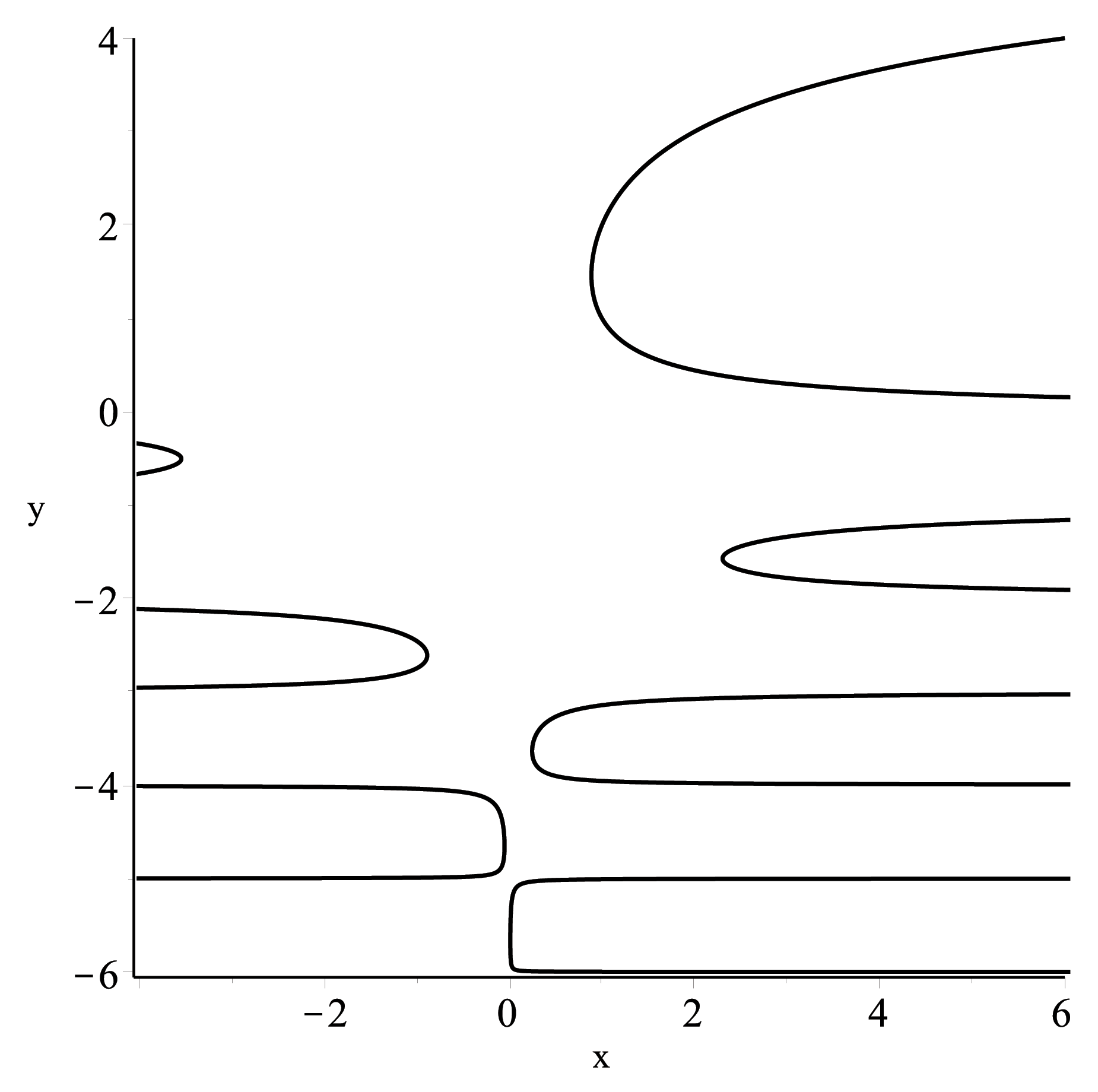}
        \caption{}
        \label{fig:realinvG}
    \end{subfigure}
    \caption{(a) A sketch of a portion of the plot of the real-valued \gf showing some singularities. The plot was produced in Maple 2016 using the code archived at \protect\url{http://www.apmaths.uwo.ca/~rcorless/frames/Gamma/GammaBorweinCorless.mw}. (b) A sketch of a portion of the functional inverse of the real-valued \gf showing several branches including the largest. The plot was produced using Maple 2016. The code is archived at \protect\url{http://www.apmaths.uwo.ca/~rcorless/frames/Gamma/GammaBorweinCorless.mw}}
\end{figure*}

In figure~\ref{fig:realgammapara} we see several locations where $\psi (x) = \Gamma'(x) /\Gamma (x)$ is zero: between $x=1$ and $x=2$, (at about $x_{0}=1.4616\ldots$), between $x=-1$ and $x=0$ (at about $x_{-1}=-0.5048\ldots$), between $x=-2$ and $x=-1$ (at about $x_{-2}=-1.5734\ldots$), and so on. Denote $x_{-k}$ the root between $-k$ and $-k+1$, except $x_{0}\in (1,2)$. Each of these locations separates a pair of branches of~$\invG$ in figure~\ref{fig:realinvG}. The branch with~$\invG(x) \geq x_{0}$ is of special interest, because it contains the inverse factorials: $\Gamma (n)=(n-1)!$ means that $\invG(24)=5$, for instance, on this branch. We take up the asymptotics of this branch of $\invG$ in section~\ref{sec:stirling}.

\subsubsection{Plotting the Riemann Surface for the Inverse $\Gamma$ Function}

Given a reliable routine to compute a complex function, it is possible to visualize the Riemann surface for inv$f$ automatically~\cite{corlessjeffrey1998graphing,trott2002,trott1997}. The basic idea is the same as above. If $z=x+iy$ and $f(z)=u+iv$, then a Riemann surface for inv$f$, the functional inverse of $f$, can be plotted by plotting $[u,v,y]$ as a parametric plot with $x$ and $y$ as parameters. Note that a routine to compute inv$f$ is not required.

There is one theoretical consideration, namely how to ensure that each point on the representation corresponds uniquely to a point on the four-dimensional map $(x,y)\leftrightarrow(u,v)$, which isn't always possible because not all Riemann surfaces are embeddable in 3D, but this can be mitigated by the use of color as the fourth dimension. There is a practical consideration, as well: namely how to choose the portion of the surface to fit into the viewing frame. For the $\Gamma$ function, there is a further difficulty, namely its extreme dynamic range, even in its imaginary part. 

Because the functional inverse of the reciprocal of $f$ is related simply to the functional inverse of $f$, as follows,
\begin{align}
y = f(x) &\Longleftrightarrow x=\mathrm{inv}f(y)\\
q = 1/f(p) &\Longleftrightarrow p=\mathrm{inv}f(1/q), \phantom{due to} q \neq 0
\end{align} we decided to plot a portion of the Riemann surface for $\invG (1/z)$. See figure~\ref{fig:riemann}. The multilayer convolutions show a rich structure for investigation. Our code is archived at \url{http://www.apmaths.uwo.ca/~rcorless/frames/Gamma/GammaBorweinCorless.mw}.

\begin{figure}[htb!]
	\centering
	\includegraphics[scale=0.4]{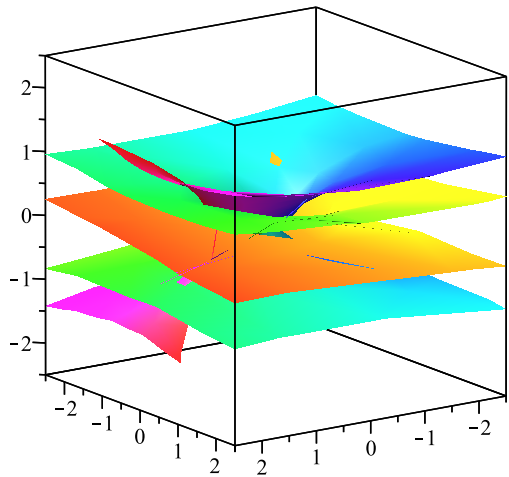}
	\caption{A visualization of a portion of the Riemann surface for inv$\Gamma (1/z)$, showing a rich multilayered structure. This figure may be misleading---more investigation is needed.}
	\label{fig:riemann}
\end{figure}

As a practical matter, constructing \textit{branches} of $\invG$ that people can use seems to be at least as important. A colleague and RMC have completed some preliminary investigations, but we'll leave reporting on this to future work. We note the foundations provided by~\cite{pedersen2015}.

\section{Stirling's Formula}\label{sec:stirling}

Judging by the number of articles in the Monthly on the subject, Stirling's formula approximating $n!$ for large $n$ is by far the most popular aspect of the $\Gamma$ function. There are ``some remarks'', ``notes'', more ``remarks''; there are ``simple proofs'', ``direct proofs'', ``new proofs'', ``corrections'', ``short proofs'', ``very short proofs'', ``elementary'' proofs, ``probabilistic'' proofs, ``new derivations'', and (our favourite title) ``The (n+1)th proof''. The earliest one we can find is~\cite{hummel}; the paper~\cite{ross} is about a different formula\footnote{Which could be applied to $z!$ and $\log\Gamma(1+z)$. See~\cite{davis}.}.\\

The name ``\stf'' turns out not to be historically correct, as we will see. However, there will be no changing the convention. Here is what we mean conventionally by \stf: 
\begin{align}
\ln \Gamma(z) &= \bigg( z-\frac{1}{2} \bigg)\ln z - z + \ln \sqrt{2\pi} + r  \label{eq:L}
\end{align} where $\Gamma$ is asymptotic to a divergent series
\begin{align}
r &\sim \sum_{n \geq 1} \frac{B_{2n}}{2n(2n-1)z^{2n-1}} = \frac{1}{12z} + \frac{7}{360z^{3}} + \cdots \>, \label{r}
\end{align} and the $B_{2n}$ are Bernoulli numbers, with generating function
\begin{equation} \label{eq:bernoulli}
\frac{t}{\e^{t}-1} = \sum_{n \geq 0} \frac{B_{n}}{n!}t^{n} \>.
\end{equation}

Stirling's formula also occurs in an exponentiated form, $\Gamma (z) \sim \sqrt{2\pi}z^{z-\sfrac{1}{2}}\e^{-z}\e^{r}$ and one can easily compute a series for $\e^{r}$ from~(\ref{eq:bernoulli}). It turns out that
\begin{align} \label{eq:e^r}
\e^{r} &\sim 1 + \frac{1}{12z} + \frac{1}{288z^{2}} - \frac{138}{51840z^{3}} + \cdots = 1 + \sum_{k\geq 2} \frac{1 \cdot 3 \cdots (2k-1)a_{2k-1}}{z^{k-1}}
\end{align} with $a_{1}=1$ and for $n>1$
\begin{equation} \label{eq:oddN}
a_{n} = \frac{1}{n+1} \bigg( a_{n-1} - \sum^{n-1}_{\ell = 2} \ell a_{\ell} a_{n+1-\ell} \bigg) \>.
\end{equation} Note that only the odd $n$ occur in~(\ref{eq:e^r}) but that even $n$ \textit{are} needed in the recurrence.

Most people don't bother to distinguish equation~(\ref{eq:L}) from its exponential: both are useful, and both are called Stirling's formula or Stirling's series. Note that equation~\eqref{eq:L} contains only odd powers and thus seems more efficient~than equation~\eqref{eq:e^r}: $\e^{r}$ contains all powers. Exponentiation of a series is simple. See e.g.~\cite[ch.~2]{corlessfillion}.

\subsection{Trapezoidal Rule analyses}

Hummel motivates the paper~\cite{hummel} by claiming 
\begin{quotation}
	``The average college undergraduate, with nothing beyond an elementary calculus course, knows little about infinite series and still less about infinite products, hence gets very little out of one of the usual developments of \stf.''
\end{quotation}

Students were different, then, as was the curriculum. Nevertheless the conclusions may still be true. Hummel then goes on in~\cite{hummel} to use the concavity of logarithm and its consequential inequality from the trapezoidal rule to establish 
\begin{equation} \label{eqn:hummel}
\e^{11/12}\sqrt{n}\hspace{1mm}n^n\e^{-n} < n! < \e\sqrt{n}\hspace{1mm} n^n \e^{-n}
\end{equation} and claims $\e^{11/12} \doteq \sqrt{2\pi} ;$ since $\e^{11/12} \doteq2.5009$ while $\sqrt{2\pi} \doteq 2.5066$ this claim is entirely reasonable. The upper bound is not very tight.

Several Monthly papers such as Nanjundiah's~\cite{nanjundniah} and the highly-cited~\cite{robbins} by Robbins use a similar starting point, though few refer to~\cite{hummel}. The paper~\cite{zeitlin} uses it \textit{backwards} to show Stirling's formula implies the integral for $\ln x$. Using the trapezoidal rule \textsl{is} a natural idea: the trapezoidal rule gives a lower bound on an integral of a concave function. The paper~\cite{johnsonbaugh} uses the trapezoidal rule to establish the existence of the limit---this is not a petty point. The paper~\cite{coleman} uses an average of trapezoids to get both a lower bound and an upper bound for $n!$.

Here is the basic element of the trapezoidal analysis: because $\ln t$ is concave, the trapezoidal rule gives a lower bound for the integral:
\begin{equation}
\frac{1}{2} \ln j + \frac{1}{2} \ln (j+1) < \int^{j+1}_{t=j} \ln t \ dt \>.
\end{equation} We now do our own version. Adding up panels of width $1$ from $t=m$ to $t=n$, where $m<n$,
\begin{align}
\frac{1}{2} \ln m + \sum^{n-1}_{j=m+1} \ln j + \frac{1}{2} \ln n < \int^{n}_{t=m} \ln t \ dt = n \ln n - n -m \ln m + m \>.
\end{align} Adding $\sum^{m}_{j=1} \ln j - \ln(m)/2 + \ln(n)/2$ to both sides,
\begin{equation}
\sum^{n}_{j=1} \ln j < n \ln n + \frac{1}{2} \ln n - n + \sum^{m}_{j=1} \ln j - \frac{1}{2} \ln m - m \ln m + m
\end{equation}or 
\begin{equation}
\ln n! < \bigg(n+\frac{1}{2} \bigg) \ln n - n + \ln d_{m}
\end{equation} where
\begin{equation}
d_{m} = \frac{m!\e^{m}}{m^{m+\sfrac{1}{2}}} \>.
\end{equation} Exponentiating,
\begin{equation} \label{eq:exponentiating}
n! < d_{m} n^{n+\sfrac{1}{2}}\e^{-n} \>.
\end{equation}

We've thus replaced the constant $c=\sqrt{2\pi}$ in \stf $\>$ by a computable upper bound. The larger we take $m$, the tighter the bound. When $m=1$ this is the upper bound quoted in~\cite{hummel}. Taking $m=n-1$ gives the tightest possible estimate, but is silly, because to compute $d_{n-1}$ requires $(n-1)!$. We will consider $m \geq n$ later. Taking $m$ small or of moderate size is reasonable in order to estimate $n!$ for large $n$. Taking $m=10$ gives $d_{10} \doteq 2.5276$, whereas $\sqrt{2 \pi} \doteq 2.5066$. Taking $m=20$ gives $d_{20} \doteq 2.5171$. More work shows $d_{m}$ monotonically decreases. 

\subsubsection{Binet's Formula}

We pass so easily from $n!$ to $z!$ nowadays that it's hard to notice the steps in passing from the asymptotics of $n!$ to the asymptotics of $z!$, and of course it's obvious that the formula must be the same\footnote{Some proofs do appear in the Monthly, e.g. Blyth and Pathak~\cite{blyth} use $\Gamma (n+\alpha) \sim n^{\alpha} \Gamma (n)$ to do so.}. Nonetheless , it's satisfying to read in~\cite{whittaker} a proof using Binet's formula, namely (for $\operatorname{Re}z > 0$)
\begin{equation} \label{eq:binet}
\ln \Gamma (z) = \bigg( z - \frac{1}{2} \bigg) \ln z - z + \ln \sqrt{2\pi} + \int_{t=0}^{\infty} \bigg( \frac{1}{2} - \frac{1}{t} + \frac{1}{\e^{t} - 1} \bigg) \frac{e^{-tz}}{t} \ dt \>,
\end{equation} that they really are the same. In the Monthly, we find~\cite{sasvari} which gives a similar proof of~(\ref{eq:binet}) and an intelligible recap of its use to derive the logarithmic variant (\ref{eq:L}) of Stirling's formula, plus bounds guaranteeing the errors in truncation.

\subsection{Other papers}

One of the prettiest papers is~\cite{aissen}. The author, M. I. Aissen, gives a concise but lucid exposition that uses the simple inequality $n!<n^{n}$ and the properties of the ratios $u_{n}=n^{n}/n!$ to make the appearance of $\e$ seem inevitable and completely natural. From there the author gives two proofs of Stirling's formula, up to $\sqrt{2\pi}$.

The shortest paper that we found is~\cite{pippenger1980}: in just 17 lines, including title, references, and address, Pippenger establishes the following ``striking companion to Wallis' product"
\begin{equation}
\frac{\e}{2} = \left( \frac{2}{1} \right)^{1/2} \left( \frac{2}{3} \frac{4}{3} \right)^{1/4} \left( \frac{4}{5} \frac{6}{5} \frac{6}{7} \frac{8}{7}  \right)^{1/8} \cdots
\end{equation} by appeal to Stirling's formula.

The papers~\cite{nanjundiah} and~\cite{maria} give simple refinements of the result from the previously mentioned~\cite{robbins} by Robbins.

The author of~\cite{feller1967}, a famous statistician from Princeton, was ``deeply apologetic" in~\cite{feller1968} for various errors. We point this out not in a spirit of schadenfreude but rather in sympathy; we hope there are no errors in this present work, but.

In~\cite{patin} we find a ``very short proof" that eschews the Central Limit Theorem as claimed to be used in~\cite{blyth} because that ``cannot reasonably be considered elementary." Instead, the author (Patin) uses the Lebesgue dominated convergence theorem! In fact the authors of~\cite{blyth} criticize earlier papers, namely~\cite{khan} and~\cite{wong1977}, for inappropriately using the Central Limit Theorem which they say gives results valid only for the asymptotics of $n!$ for integers $n$. They use moment generating functions instead.

The paper~\cite{pinsky2007} uses the Poisson distribution to get Stirling's formula (including the $\sqrt{2\pi}$). The paper~\cite{lou2014} uses change of variables. The pair of papers by Reinhard Michel~\cite{michel2002} and~\cite{michel2008} use artful estimates of $\ln$ and skilful substitution to get several terms of Stirling's formula in a really elementary way. The second paper has our favourite title, ``The $(n+1)$th proof of Stirling's formula", and takes a swipe at Patin for calling the Lebesgue dominated convergence theorem ``elementary". The letter to the Editor~\cite{nelsen1989} takes a different swipe and claims Patin's proof is contained in Newman's Problem Seminar.

Six years after the $(n+1)$th proof, we have another ``new proof" in~\cite{neuschel}. 

The very impressive paper~\cite{diaconis1986} gives a concise but complete version of Laplace's argument, and a nice historical discussion, including mention of de Moivre.

The paper~\cite{mumma1986} gives an elementary proof that
\begin{equation}
\lim_{n \to \infty} \frac{(n!)^{1/n}}{n} = \frac{1}{\e}
\end{equation} and thereby extends the power of the root test.

The paper~\cite{namias1986} did something quite different, namely use the Legendre duplication formula~(\ref{eq:Legendre}) and generalization to get some new recurrence relations for the logarithmic variant, equation~(\ref{eq:L}). This generated the follow-up paper~\cite{deeba1991} because this gave new formulae for Bernoulli numbers. 

The paper~\cite{eger} extends Stirling's formula for use with a certain generalization of central binomial coefficients. Its focus seems to be statistical, and it uses discrete random variables and the Central Limit Theorem. 

The paper~\cite{marsaglia} is not without flaws, and indeed was criticized in~\cite{grossman} for not being as ``new'' as its title claimed\footnote{We found~\cite{grossman} by browsing the Monthly, not by citation search. Its criticism of~\cite{marsaglia} is largely sound: Copson's treatment is better in many respects. }, having been anticipated by the textbook treatment in~\cite{copson}. Indeed,~\cite{copson} points further back, to Watson and Ramanujan for the main idea. Nonetheless~\cite{marsaglia} is an interesting and readable paper, and the recurrence relation developed there, which we have already presented in equation~(\ref{eq:oddN}), is apparently not in~\cite{copson} and seems to be new to the paper. This paper was discussed in some detail in~\cite{BorweinCorless1999}.

One can say more, nowadays: in~\cite{brassesco} and~\cite{nemes} we find ``explicit" formulas in terms of 2-associated Stirling numbers, and David Jeffrey reports (personal communication) the following expression:
\begin{align}
1 \cdot 3 \cdot \cdot 5 &\cdots (2n-1)a_{n}= \nonumber\\ 
&\bigg( \frac{n}{2} +1 \bigg)^{\overline{n-1}} \sum^{n-1}_{k=1} \frac{(-1)^{k}}{(n-k-1)!} \frac{2^{k}}{(n+2k)!} \begin{bmatrix} n+2k-1 \\ k \end{bmatrix}_{\geq 2} \>.
\end{align} Here $\begin{bmatrix} m \\ n \end{bmatrix}_{\geq2}$ means a $2$-associated Stirling cycle number. It is gratifying to see Stirling numbers in Stirling's asymptotic series. [Even if it's really due to de Moivre].

\subsection{Gaps: midpoint rule, Stirling's original formula, and asymptotics of $\invG$} 

\mbox{}

We saw in a previous section that several papers used the trapezoidal rule to estimate $\ln n!$ by $\int^{n}_{1} \ln t \ dt$; we used this to give a parameterized upper bound, namely equation~(\ref{eq:exponentiating}).

Nowhere in the Monthly that we have seen does the \textsl{midpoint rule} appear to be used to give a lower bound for $n!$. The following theorem rectifies this omission.
 
\begin{theorem}

 If $1 \leq k \leq n-1$, and
\begin{equation}
c_{k} = \frac{k ! \e^{k + \sfrac{1}{2}}}{(k + \sfrac{1}{2})^{k + \sfrac{1}{2}}} \>,
\end{equation} then
\begin{equation}
c_{k} ( n + \sfrac{1}{2})^{n+\sfrac{1}{2}}\e^{-(n+\sfrac{1}{2})} < n! \>.
\end{equation}

\end{theorem}

\begin{proof} 

For a concave function, the midpoint rule gives an upper bound\footnote{This is in several texts, including~\cite{corlessfillion}--- but one place to learn it is by teaching from~\cite{stewart2015}.}, so if $j \geq 1$
\begin{equation}
\int_{j-\sfrac{1}{2}}^{j+\sfrac{1}{2}} \ln t \ dt < \ln j \>.
\end{equation} Using panels of width 1, if $k<n$,
\begin{equation}
\int^{n+\sfrac{1}{2}}_{k+\sfrac{1}{2}} \ln t \ dt < \sum^{n}_{j=k+1} \ln j \>.
\end{equation} Adding $\sum_{j=1}^{k} \ln j$ to both sides, 
\begin{align}
(n+\sfrac{1}{2})\ln(n+\sfrac{1}{2}) - (n+\sfrac{1}{2}) &- (k+\sfrac{1}{2})\ln(k+\sfrac{1}{2}) + (k+\sfrac{1}{2})\\ 
&+ \sum^{k}_{j=1} \ln j < \sum^{n}_{j=1} \ln j \nonumber
\end{align} or
\begin{align}
\ln c_{k} + (n+\sfrac{1}{2})\ln(n+\sfrac{1}{2})-(n+\sfrac{1}{2}) < \ln n!
\end{align} Exponentiating gives the theorem. \end{proof}

\begin{remark}
We now consider what happens when $k = n$ or $m=n$: the inequalities become equality, and we learn $n!=n!$ which is not a surprise. But what if $k > n$, or $m > n$? In those cases the sense of the inequalities \textit{reverse} because the integrals change sign: 
\begin{equation}
d_{m}n^{n+\sfrac{1}{2}}\e^{-n} < n! < c_{k} (n+\sfrac{1}{2})^{n+\sfrac{1}{2}}\e^{-(n+\sfrac{1}{2})}
\end{equation} if both $m>n$ and $k > n$. One can establish also that $d_{m}$ decreases and is bounded below by $\sqrt{2\pi}$, while $c_{k}$ increases and is bounded above by $\sqrt{2\pi}$; we have
\begin{equation}
\sqrt{2\pi} n^{n+\sfrac{1}{2}}\e^{-n} < n! < \sqrt{2\pi}(n + \sfrac{1}{2})^{n+\sfrac{1}{2}}\e^{-(n+\sfrac{1}{2})} \>.
\end{equation} These bounds and the finitary bounds of the theorem, i.e. 
\begin{equation}
c_{k}\bigg( n+\frac{1}{2} \bigg)^{n+\frac{1}{2}}\e^{-(n+\frac{1}{2})} < n! < d_{m}n^{n+\frac{1}{2}}\e^{-n}
\end{equation} if $k, m<n,$ are similarly tight to those produced by examining the error term in the trapezoidal rule (e.g.~\cite{hummel}).
\end{remark}

\begin{remark}
The first draft of this paper contained the phrase ``we believe this is new, although it's difficult to be sure." Jon read that and was skeptical. It turns out that he was correct in part, but the story is interesting. Our first real clue was the paper~\cite{dutkay} which gives, at the end, an approximation containing the characteristic factor $(n-1/2)^{n-1/2}$ (for $\Gamma (n)$, not $n!$). They attribute this result to Burnside~\cite{burnside}, in $1917$ (a hundred years ago!). But the form is much older than that, even: \textsl{Stirling himself} gives (using ``$\ell$", to denote the base--10 logarithm)
\end{remark}
\begin{equation}
z \ell, z - az - \frac{a}{24z} + \frac{7a}{2880z^3} - \&c \>.
\end{equation} ``added to half the logarithm of the circumference of the circle whose radius is one" in example 2 of Proposition 28~\cite[p.~151]{tweddle2003}. This is an expression for $\log_{10}(n!)$ but here $\boldsymbol{z=n+1/2}$; the constant $a$ is $1/\ln(10)$ and this means, in modern notation,
\begin{align}
\ln n! \sim \ln\sqrt{2\pi} + \bigg( &n + \frac{1}{2} \bigg) \ln \bigg( n + \frac{1}{2} \bigg) \\
&- \bigg( n + \frac{1}{2} \bigg) - \frac{1}{24(n+1/2)} + \frac{7}{2880(n+1/2)^{3}} - \cdots \>.
\end{align} This is not equation (\ref{eq:L})! This is, as Tweddle points out in~\cite{tweddle1984}, a mid-point formula! Stirling did not explicitly use the midpoint rule or the concavity of logarithm to establish an inequality; instead he used telescoping series, induction, and a (clear and algorithmic) recipe for computing more terms to give \textit{his} series. Apparently, de~Moivre was the first to give the simpler version that we call ``Stirling's formula" today, namely equation (\ref{eq:L}).

The following theorem can be proved using very similar steps as equation~(\ref{eq:binet}) is proved in~\cite{whittaker} and~\cite{sasvari} (and is rather fun to do). We leave this for the reader. 

\begin{theorem}
For $\operatorname{Re}z > 1/2$,
\begin{align}
\ln \Gamma(z) &= (z-1/2) \ln (z-1/2) - (z-1/2) + \ln \sqrt{2\pi} \nonumber \\ 
&+ \int^{\infty}_{t=0} \frac{1}{t} \bigg( \frac{1}{t} - \frac{1}{2\sinh t/2} \bigg) \e^{-(z-1/2)t} \ dt \>.
\end{align}
\end{theorem} This should be contrasted to equation~(\ref{eq:binet}); the known series expansion for $\csch (z)$ (see e.g.~\citep{abramowitz1966}) can then be used to fully develop Stirling's original formula, after an appeal to Watson's lemma~\cite{bender1999}:
\begin{align}
\int^{\infty}_{t=0} \frac{1}{t} \bigg( \frac{1}{t} &- \frac{1}{2 \sinh t/2} \bigg) \e^{-(z-1/2)t} \ dt \nonumber \\ 
&\sim \sum_{n \geq 1} \frac{(1-2^{1-2n})}{2n(2n-1)} \frac{B_{2n}}{(z-1/2)^{2n-1}} \\
&\sim \frac{-1}{24(z-1/2)} + \frac{7}{2880(z-1/2)^{3}} - \frac{31}{40320(z-1/2)^{5}} + \cdots
\end{align}

Burnside~\cite{burnside} gave, by entirely different means, the formula
\begin{align}
\ln \Gamma (N+1) &= (N+1/2)\ln(N+1/2)-N+\ln \sqrt{\frac{2\pi}{\e}} \nonumber\\
&- \sum_{r \geq 1} \sum_{n \geq N +1} \frac{1}{2r(2r+1)2^{2r}n^{2r}} \>,
\end{align} which, being convergent, gives therefore a double sum for this Binet-like integral when $-1+z=N$ is an integer. Burnside did not comment on the similarity to Stirling's original formula. He did, however, compute a constant
\begin{equation}
C = \frac{1}{2} \ln 2 + \sum_{r \geq 1} \sum_{n \geq 1} \frac{1}{2r(2r+1)2^{2r}n^{2r}}
\end{equation} as ``to six figures" as $0.418938$, and then identified it as $\ln \sqrt{2\pi/\e}=0.418938533\ldots$ using ``Stirling's Theorem", but using $N^{N+1/2}$ not $(N+1/2)^{N+1/2}$. This suggests that he did not realize that his formula was actually closer to Stirling's original.

\subsubsection{Asymptotics of inv$\Gamma$}

The first few terms of Stirling's original formula give
\begin{equation}
\ln \Gamma (z) = \bigg( z- \frac{1}{2} \bigg) \ln \bigg( z- \frac{1}{2} \bigg) - \bigg( z - \frac{1}{2} \bigg) + \ln \sqrt{2\pi} + \mathcal{O}\bigg( \frac{1}{z} \bigg) \>.
\end{equation} We wish to solve $x=\Gamma (y)$ for $y$ when $x$ is large.

Put $v=x/\sqrt{2\pi}$ and $u = y -  1/2$. Then the formula above gives
\begin{equation}
\ln v = u \ln u - u + \mathcal{O} (1/u) \>.
\end{equation} Ignoring the $\mathcal{O}(1/u)$ term, we have
\begin{equation}
\ln v \doteq u \ln u - u =u(\ln u -1) = u \ln \frac{u}{\e}
\end{equation} and, somewhat remarkably, we can solve this for $u$, as follows:
\begin{align}
\frac{\ln v}{\e} &= \frac{u}{\e} \ln \frac{u}{\e} = \bigg( \ln \frac{u}{e} \bigg) \e^{\ln u/ \e} \>
\end{align} so (using the Lambert $W$ function)
\begin{align}
\ln \frac{u}{\e} &= W \bigg ( \frac{\ln v}{\e} \bigg) \>.
\end{align} Exponentiating,
\begin{align}
\frac{u}{\e} &= \e^{W(\ln v/\e)} = \frac{\ln v/\e}{W\bigg( \frac{\ln v}{\e} \bigg)} \>. \\
\shortintertext{Therefore}
y_{0} &= \frac{1}{2} + \frac{\ln (x/\sqrt{2\pi})}{W \big( \frac{1}{\e} \ln (x/\sqrt{2\pi}) \big)} \>. \label{eq:asymptInvGo}
\end{align} This formula approximately inverts $\Gamma$, so $x \doteq \Gamma (y_{0})$. Using James Stirling's own example, $1000!$ or $\Gamma (1001) \doteq 4.0238 \cdot 10^{23567}$, the formula above gives $1000.99999397$, a relative error of about $6 \cdot 10^{-9}$, from the computed value of $1000!$. 

The branch of $W$ used above is the principal branch, $W_{0}(x)$, defined for $-1/\e \leq x$ and satisfying $W_{0}(x) \geq -1$. Thus $\ln(x/\sqrt{2\pi})/\e \geq -1/\e$, or $x \geq \sqrt{2\pi}/\e$.
\\
\noindent At $x=\sqrt{2\pi}/\e$, $y_{0} = 1/2 - 1/W_{0}(-1/\e) = 1/2 + 1 = 3/2$. This states that $x=\Gamma (y_{0})$ should be approximately true when $x=\sqrt{2\pi}/\e$ and $y_{0}=3/2$. We can check this: $\Gamma(3/2) = \Gamma(1/2)/2=\sqrt{\pi}/2$. Now $\sqrt{2\pi}/\e \doteq 0.922137$, $\sqrt{\pi}/2 \doteq 0.88622$. These differ only by about 4\%. Note that this accuracy is attained at the lowest possible value of $x$ for which the formula is valid at all, and that (as we shall see) the error decays reasonably rapidly as $x \to \infty$. This is very nearly a universal formula. Of course this is because Stirling's formula is so accurate\footnote{We remind the reader that divergent series are often remarkably accurate. There are two relevant limits, only one of which might be bad. However, Stirling's formula is famously accurate, more so than most. And Stirling's \textit{original} formula is even better, because the midpoint rule is twice as accurate.}.

But it gets better. For $-1/\e \leq \ln(x/\sqrt{2\pi})/\e < 0$, we may use $W_{-1}(x)$ as well. That is, on
\begin{equation}
\frac{\sqrt{2\pi}}{\e} \leq x < \sqrt{2\pi}
\end{equation} the formula
\begin{equation} \label{eq:asymptInvGM}
y^{(-)} = \frac{1}{2} + \frac{\ln(x/\sqrt{2\pi})}{W_{-1} \bigg( \frac{1}{\e} \ln \bigg(\frac{x}{\sqrt{2\pi}} \bigg) \bigg)}
\end{equation} also inverts Stirling's original formula. Plotting, on the same graph as $\invG$, we get the picture in figure~\ref{fig:asymptInvG}. The agreement is astonishing.

\begin{figure}[htb!]
	\centering
	\includegraphics[scale=0.35]{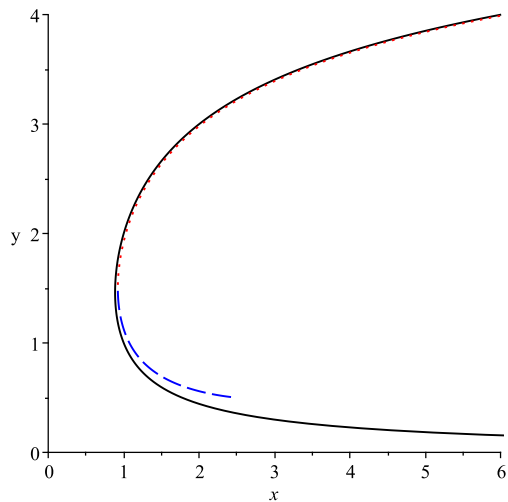}
	\caption{A plot showing that inverting Stirling's original approximation for $n!$ gives an astonishing accuracy even for small $x$; by using both real branches of the Lambert~$W$ function, the asymptotically accurate approximation even correctly (qualitatively) gets the ``turning the corner", up from $x=\sqrt{2\pi}/\e$ to $x=\sqrt{2\pi}$. The inverse Gamma function~$\invG$ is in black (solid line), the fromula~(\ref{eq:asymptInvGo}) in red (dotted line), and the other branch from~(\ref{eq:asymptInvGM}) in blue (dashed line).}
	\label{fig:asymptInvG}
\end{figure} 

We claim but do not prove here (this paper is already too long) that the relative error in $y$ is $\mathcal{O}(1/y^{2})$ as $x \to \infty$; more precisely, the next two terms in the series are
\begin{equation}
y = \frac{1}{2} + u_{0} + \frac{1}{24u_{0}(1+w)} - \frac{5 + 10(1+w)+14(1+w)^{2}}{5760(1+w)^{3}u_{0}^{3}} + \cdots
\end{equation} where $u_{0}=\ln(x/\sqrt{2\pi})/w$ with $w=W_{0}(\ln(x/\sqrt{2\pi})/\e)$. There is a surprising connection to one of the series in~\cite{wang2016}, but this will be discussed another day.

We remark that de Moivre's simpler formula 
\begin{equation}
\ln \Gamma (z) \sim z \ln z - \frac{1}{2} \ln z - z + \ln \sqrt{2\pi}
\end{equation}
is not as simple to invert:
\begin{equation}
\ln \bigg( \frac{x}{\sqrt{2\pi}} \bigg) \doteq \bigg(y-\frac{1}{2} \bigg) \ln y - y
\end{equation} doesn't have a solution in terms of Lambert $W$, and thus we're forced to approximate:
\begin{equation}
\ln \bigg( \frac{x}{\sqrt{2\pi}} \bigg) \doteq y \ln y - y - \frac{1}{2} \ln y \>,
\end{equation} and ignoring the $-1/2 \ln y$, we get the same equation as before, but now 
\begin{equation}
y = \frac{\ln (x / \sqrt{2\pi})}{W \bigg(\frac{1}{\e} \ln x/\sqrt{2\pi} \bigg)}
\end{equation} which always differs by $1/2$ from the previous.

\subsection{Drawing the line somewhere}

We could have included here a look at~\cite{axness},~\cite{fine1947},~\cite{goetgheluck1987}, \\~\cite{golomb1980},~\cite{golomb1992},~\cite{konvalina2000},~\cite{ogilvy},~\cite{wolfram1984}, or for that matter~\cite{knuth1992}, but no. Read them yourselves!

\section{Concluding Remarks}

This has been an amazing journey for us, through the Monthly and through a selection of the vast $\Gamma$ literature. The $\Gamma$ bibliography at milanmerkle.com had $986$ entries when we accessed it. This present survey has focussed on Monthly papers, with some others for context. Diego Dominici's survey of Stirling's formula~\cite{dominici} attempted some of this, and we found his paper (in the Nota di Matematica) helpful; but indeed this whole project was a wonderful excuse to read great mathematical works. The ostensibly novel contributions of this paper (the pictures of the Riemann surface for inv$\Gamma$, the asymptotic formula for inv$\Gamma$, and the elementary bracketing inequality~(\ref{eqn:hummel}) for $n!$ that has approximations to Stirling's original formula on the left and de Moivre's simpler formula that we call ``Stirling's formula" on the right), are small in comparison.

Stirling's original series for $n!$ is largely unknown, although the historians have it right; and Wikipedia correctly attributes the series popularly known as ``Stirling's formula" to de Moivre. For more information on de Moivre's contributions, see~\cite{bellhouse2011}. Recently, Peter Luschny~\cite{luschny2012} and Weiping Wang~\cite{wang2016} have started calling the exponentiated version of Stirling's original formula by the name ``the de Moivre formula"~\cite[pg.~584]{wang2016}! In view of similar historical misnamings (Euler invented Newton's method; Newton invented (symplectic) Euler's method, for instance~\cite{wanner2010}), we're ok with this.

After writing this paper, we discovered that David W. Cantrell in 2001~\cite{cantrell2001} found the asymptotic formula for $\invG$, before us. However, we then found that it's also in~\cite{Gonnet1981}, beating us all by decades!\\

\textit{Dedicated to the memory of Jonathan M. Borwein.} 

\section*{About the Paper}
I proposed the idea of this paper to Jon in late May 2016, and he was immediately enthusiastic. Jon suggested that we start by scouring JSTOR for Monthly titles containing ``Gamma'' or ``Stirling's formula''. Later, we added the terms ``factorial'' and ``psi'' to the search. I wrote the first draft of the paper, containing the elements of every section except section~\ref{sec:gammaitself}, meaning to leave that to Jon. In our many lunchtime discussions of the evolving draft over the months of June and July, Jon made many contributions, including pointing out connections to some of his early work. He also suggested that one of my results in section~\ref{sec:stirling} was not as `new' as I had thought; as I now know, Jon was right.

Jon died early in the morning August 2, 2016, the week I was to pass the draft over to him. Jon's influence on this paper is therefore much less than it could have been. Now we'll just have to imagine what Jon would have done with section~\ref{sec:gammaitself}. 

On a final note, the task that this paper attempts is impossible in reasonable time (and we both knew it). A sample of the Problems and Solutions section turns up $\Gamma(z)$ or $n!$ or $\psi$ more frequently than not, as the simplest rebuttal of a claim to completeness. But I like to think that this paper achieves 80$\%$ coverage. Had Jon lived, we might have achieved 90$\%$. 

Judith Borwein has graciously permitted Jonathan's name to be listed as author; certainly I feel it's justified. It wasn't just details: part of this paper was his vision.
\begin{acknowledgment}{Acknowledgment.}
This work was supported by Western University through a Distinguished Visiting Research Fellowship for the first author, and by NSERC, the Fields Institute for Research in the Mathematical sciences, the Rotman Institute for Philosophy, and the Ontario Research Centre for Computer Algebra for the second. The manuscript was typed by Julia E. Jankowski and Chris Brimacombe. 
\end{acknowledgment}

\bibliographystyle{myplainnat}
\bibliography{bib.bib}

\end{document}